\documentclass{amsart}
\usepackage{amsmath,amsthm,amssymb,xypic}
\usepackage[all]{xy}

\title{Equivalent Birational Embeddings III: cones}

\author{Massimiliano Mella}
\address{Dipartimento di Matematica e Informatica\\
Universit\`a di Ferrara\\
Via Machiavelli 35\\
44100 Ferrara, Italia} \email{mll@unife.it}

\date{October 2012}
\subjclass{Primary 14E25 ; Secondary 14E05, 14N05, 14E07}
\keywords{Birational maps; Cremona equivalence; embeddings;
hypersurfaces}
\thanks{Dedicated  to Alberto Conte in
  the occasion of his $70^{\rm th}$ birthday.}

\theoremstyle{plain}
\newtheorem{Th}{Theorem}

\newtheorem{theorem}{Theorem}[section]

\newtheorem{proposition}[theorem]{Proposition}
\newtheorem{lemma}[theorem]{Lemma}
\newtheorem{corollary}[theorem]{Corollary}

\theoremstyle{definition}

\newtheorem{definition}[theorem]{Definition}

\theoremstyle{remark}

\newtheorem{remark}[theorem]{Remark}

\DeclareMathOperator{\mult}{mult}

\newcommand{\QED}{\ifhmode\unskip\nobreak\fi\quad {\rm Q.E.D.}} 

\newcommand\map{\dasharrow}
\newcommand\iso{\cong}

\newcommand{\f}{\varphi}

\renewcommand{\H}{\mathcal{H}}

\newcommand{\I}{\mathcal{I}}

\renewcommand{\O}{\mathcal{O}}
\renewcommand{\P}{\mathbb{P}}

\newcommand{\Q}{\mathbb{Q}}

\newcommand{\rat}{\dasharrow}

\begin{document}

\begin{abstract} Two divisors in $\P^n$
are said to be Cremona equivalent if
there is a Cremona modification sending
one to the other. In this paper I study irreducible cones in $\P^n$ and prove
that two cones are Cremona equivalent if  their general
hyperplane sections are birational.
In particular I produce examples of cones in $\P^3$ Cremona
equivalent to a plane whose plane section is not Cremona equivalent to a
line in $\P^2$.
\end{abstract}

\maketitle

\section*{Introduction}
Let $X\subset\P^n$ be an irreducible and reduced projective variety over an algebraically
closed field. A classical question is to study the birational
embedding of $X$ in $\P^n$ up to the Cremona group of $\P^n$. 
In other words let $X_1$ and $X_2$ be 
two birationally equivalent
projective varieties in $\P^n$.  One wants to understand  if there exists  a
Cremona transformation of $\P^n$ that maps $X_1$ to $X_2$, in this
case we say that $X_1$ and $X_2$ are Cremona equivalent. This
projective statement  can also be interpreted in terms of log Sarkisov
theory, \cite{BM}, and is
somewhat related to the  Abhyankar--Moh problem, \cite{AM} and
\cite{Je}. In the latter paper it is proved, using techniques derived
form A--M problem, that over the complex
field the birational
embedding is unique as long as $\dim X< \frac{n}2$. 
The problem is then completely solved in \cite{MP} where it is proved
that this is the case over any algebraically closed field as long as the codimension of $X_i$ is at
least 2. Examples of inequivalent embeddings of divisors are well
known, see also \cite{MP}, in all dimensions.
The problem of Cremona equivalence is therefore reduced to study the
equivalence  classes of divisors. This can also be interpreted as the action
of the Cremona group on the set of divisors of $\P^n$. 

The special case of plane curves received a lot of attention both in
the old times, \cite{Co}, \cite{SR}, \cite{Ju}, and in more recent times, \cite{Na}, \cite{Ii}, \cite{KM},
\cite{CC}, and \cite{MP2}, see also \cite{BB} for
a nice survey. In \cite{CC} and \cite{MP2} a complete
description of plane curves up to Cremona equivalence is given and in
\cite{CC} a detailed study of the Cremona equivalence for linear systems 
is furnished. 
In particular it is interesting to note that the Cremona equivalence
of a plane curve is dictated by its singularities and cannot be
divined without a partial resolution of those, \cite[Example
3.18]{MP2}. Due to this it is quite hard even in the plane curve case to
determine the Cremona equivalence class of a fixed curve simply by its
equation. 

The next case is that of surfaces in $\P^3$. In this set up using the
$\sharp$-Minimal Model Program, \cite{Me} or minimal model program
with scaling \cite{BCHMc}, a criterion for detecting
surfaces Cremona equivalent to a plane is given. The criterion, inspired by
the previous work of Coolidge on curves Cremona equivalent to lines
\cite{Co}, allows to determine all rational surfaces
that are Cremona equivalent to a plane, \cite[Theorem 4.15]{MP2}. 
Unfortunately, worse than in the plane
curve case, the criterion requires not only the
resolution of singularities but also a control on different log
varieties attached to the pair $(\P^3,S)$.  As a matter of fact it is
impossible to guess simply by the equation if a rational surface in
$\P^3$ is Cremona equivalent to a plane and it is very difficult in
general to determine such equivalences. 
The main difficulty comes from the condition that the sup-threshold,
see Definition \ref{def:rsup}, is positive.  
This is quite awkward and
should be very interesting to understand if this numerical constrain
is really necessary.  Via $\sharp$-MMP it is easy, see remark \ref{rem:sharp}, to reduce this
problem to the study of pairs $(T,S)$ such that $T$ is a terminal
$\Q$-factorial 3-fold with a Mori fiber structure $\pi:T\to W$ onto
either a rational curve or a rational surface, and $S$ is a smooth
Cartier divisor with $S=\pi^*D$ for some divisor $D\subset W$.
The first ``projective incarnation'' of such pairs are cones in $\P^3$.

In this work I develop a strategy to study  cones in arbitrary
projective space. If two cones in $\P^n$ are built on
varieties Cremona equivalent in $\P^{n-1}$  then also the cones are Cremona equivalent,
see Proposition \ref{pro:easy}. The expectation for arbitrary cones built on
birational but not Cremona equivalent varieties was not clear but somewhat more on
the negative side.  The result I prove is therefore quite unexpected
and shows once more the amazing power of the Cremona group of $\P^n$.

\begin{Th}
Let $S_1$ and $S_2$ be two cones in $\P^n$. Let $X_1$ and $X_2$ be 
corresponding general hyperplane sections. If  $X_1$ and $X_2$ are
birational then $S_1$ is Cremona
equivalent to $S_2$.
\end{Th}

The main ingredient in the proof is the reduction to subvarieties of
codimension 2 to be able to apply the main result in \cite{MP}. To do this I produce a special log resolution of
the pair $(\P^n,S)$ that allows me to blow down the strict transform
of $S$ to a codimension 2 subvariety. Despite the fact that this step
cannot produce a Cremona equivalence for $S$ it allows me to work out
Cremona equivalence on the lower dimensional subvariety and then lift the
Cremona equivalence.

In the special case of cones in $\P^3$ the statement can be improved
to characterize the Cremona equivalence of cones with the geometric genus of the
plane section, see Corollary \ref{cor:3}. In particular this shows
that any rational cone in $\P^3$ has positive sup-threshold, see Definition
\ref{def:rsup}.  From the $\sharp$-MMP point of view we may easily
translate 
it  as
follows.

\begin{Th}\label{th:sok} Let $S\subset\P^3$ be a rational surface. Assume that there
  is a $\sharp$-minimal model of the pair $(T,S_T)$ such that $T$ has
  a scroll structure $\pi:T\to W$  onto a rational surface $W$ and
  $S_T=\pi^*C$, for some rational curve $C\subset W$. Then $\overline{\rho}(T,S_T)=\overline\rho(\P^3,S)>0$.
\end{Th}

The next candidate for the sup-threshold problem are pairs whose $\sharp$-minimal model
is a conic bundle and the surface is trivial with respect to the conic
bundle structure. With the technique developed in this paper I am only
able to treat a special class of these, see Corollary \ref{cor:cb}.

\section{Notations and preliminaries}
I work over an algebraically closed field of characteristic zero.
I am inte\-re\-sted in birational transformations of log pairs.
For this I introduce the following definition.

\begin{definition}\label{def:cremona} Let $D\subset X$ be
  an irreducible and  reduced divisor on
  a normal variety $X$.
We say that $(X,D)$ is birational
  to $(X',D')$, if there exists a
  birational map $\f:X\map X'$ with
  $\f_*(D)=D'$.
Let $D, D'\subset\P^n$ be
   irreducible reduced divisors then we say that $D$ is Cremona
  equivalent to $D'$ if
  $(\P^n,D)$ is birational to $(\P^n,D')$.
\end{definition}

Let us proceed recalling a well known class of singularities.

\begin{definition}
Let $X$ be a normal variety and $D=\sum d_i D_i$ a $\Q$-Weil divisor, with $d_i\leq 1$.
Assume that $(K_X+D)$ is $\Q$-Cartier. Let $f:Y\to X$ be a log
resolution of the pair $(X,D)$ with
$$K_Y=f^*(K_X+D)+\sum a(E_i,X,D) E_i$$
We call
\begin{multline*}
disc(X,D):=\min_{E_i}\{a(E_i,X,D)|\\
\mbox{$E_i$ is an $f$-exceptional
  divisor for some log resolution} \}
\end{multline*}

Then we say that $(X,D)$ is
 \[\begin{array}{ccc}
 \left.\begin{array}{l}
terminal\\
canonical\\
\end{array}
 \right\}&\mbox{if $disc(X,D)$}&
\left\{\begin{array}{l}
>0\\
\geq 0\\
\end{array}
 \right.
\end{array}
\]
\end{definition}

\begin{remark} Terminal surfaces are smooth, this is essentially the
  celebrated Castelnuovo theorem. Any log resolution of
  a smooth surface can be obtained via
  blow up of smooth points. Hence a pair $(S,D)$, with
  $S$ a smooth surface has canonical
  singularities if and only if $\mult_p
  D\leq 1$ for any point $p\in S$. 

Note further that one direction is true
in any dimension. Assume that $X$ is
smooth and  $\mult_p
D\leq 1$ for any point $p\in X$. Let  $f:Y\to X$ be a
smooth blow up, with exceptional divisor
$E$. Then $K_Y=f^*(K_X)+aE$ for some
positive integer $a$ and $a(E,X,D)\geq
a-1\geq 0$. This proves that 
 $(X,D)$ has canonical
singularities if $X$ is smooth and $\mult_pD\leq
1$ for any $p\in X$. This simple observation allows to produce many
inequivalent embeddings of divisors, see \cite[\S 3]{MP}
\label{rem:mult}
\end{remark}

For future reference we recall  a
technical result on pseudoeffective
divisors, i.e. the closure of effective divisors.
\begin{lemma}[{\cite[Lemma 1.5]{MP2}}]\label{lem:psef}
  Let $(X,D_X)$ and $(Y,D_Y)$ be
  birational pairs with canonical
  singularities. 
Then $K_X+D_X$ is pseudoeffective if and
  only if $K_Y+D_Y$ is pseudoeffective.
\end{lemma}

 The main difficulty to study Cremona equivalence in $\P^r$ with
 $r\geq 3$ is the poor knowledge of the Cremona
group.
The case of surfaces in $\P^3$ is already
quite mysterious. It is easy to show
that Quadrics and rational cubics are Cremona
equivalent to a plane. Rational
quartics with either 3-ple or 4-uple
points are again easily seen to be
Cremona equivalent to planes, the latter
are cones over rational curves Cremona
equivalent to lines. 
 It has been expected that
Noether quartic should be the first
example of a rational surface
not Cremona equi\-va\-lent to a plane, but this is not the case as
proved in \cite[Example 4.3]{MP2}.
Having in mind  these examples and  the $\sharp$-MMP
developed in \cite{Me} for linear systems on uniruled 3-folds, I
recall the definition of
(effective) threshold.
\begin{definition}
Let $(T,H)$ be a terminal $\Q$-factorial uniruled
   variety and $H$ an irreducible and
   reduced Weil divisor
  on $T$. Let
  $$\rho(T,H)=:\mbox{  \rm sup   }\{m\in \Q|H+mK_T
\mbox{
  \rm is an effective $\Q$-divisor  }\}\geq 0,$$
 be the (effective) threshold of the pair
 $(T,H)$.
\end{definition}
\begin{remark} The threshold is not a
birational invariant of the pair and it
is not preserved by blowing up.
Consider
 a Quadric cone $Q\subset\P^n$ and let
$Y\to \P^n$ be the blow up of the vertex
then $\rho(Y,Q_Y)=0$, while
$\rho(\P^n,Q)>0$. 
\end{remark}

To study Cremona equivalence, unfortunately, we have to take into
account almost all possible thresholds.

\begin{definition}\label{def:rsup}
Let $(Y,S_Y)$ be a pair birational
to a pair $(T,S)$. We say that
$(Y,S_Y)$ is a good birational model if $Y$ has
terminal $\Q$-factorial singularities
and $S_Y$ is a Cartier divisor with terminal singularities. 
The sup-threshold
of the pair $(T,S)$ is
$$\overline{\rho}(T,S):=\sup\{\rho(Y,S_Y)\},$$
where the sup  is taken on good
birational models.
\end{definition}

\begin{remark} It is clear that any pair
  $(\P^n,S)$ Cremona equivalent to a
  hyperplane satisfies
  $\overline{\rho}(\P^n,S)>0$. The pair
  $(\P^n,H)$, where $H$ is a hyperplane, is a
  good model with positive
  threshold.\label{rem:pos}

Considering birationally super-rigid
MfS's one can produce examples of
pairs, say $(T,S)$, with
$\overline{\rho}(T,S)=0$. It is not
clear to me if such examples can exist
also on varieties with bigger pliability, see 
\cite{CM} for the relevant definition.  
\end{remark}
We are ready to state the characterization of  surfaces
Cremona equivalent to a plane.

\begin{theorem}[{\cite[Theorem 4.15]{MP2}}] Let $S\subset\P^3$ be an
  irreducible and reduced surface.
The following are equivalent:
\begin{itemize}
\item[a)]  $S$ is Cremona equivalent to a plane,
\item[b)]  $\overline{\rho}(T,S)>0$ and there is a good model
  $(T,S_T)$ with $K_T+S_T$ not pseudoeffective.
\end{itemize}
\label{th:3c}
\end{theorem}
\begin{remark}
  \label{rem:sharp} As remarked in the introduction the main drawback of the above criterion  is the bound on the
  sup-threshold. It is very difficult to compute it. While the
  requirement that $K_T+S_T$ is not pseudoeffective is natural and
  justified also by Lemma \ref{lem:psef}, it
  is not clear if pairs with vanishing sup
  threshold may exist on a rational 3-fold.
This naturally leads to study good models with vanishing threshold.

Let $(X,S)$ be a good pair with $X$ rational and $\rho(X,S)=0$.  Then the $\sharp$-MMP
applied to this pair may lead to a Mori fiber space $\pi:T\to W$ such that
$S_T$ is trivial with respect to $\pi$ and it is a smooth surface, see
\cite[Theorem 3.2]{Me} and the proof of \cite[Theorem 4.9]{MP2}. In
particular $S_T=\pi^*D$ for some irreducible divisor $D\subset W$.
Then if $W$ is a curve $S_T$ is a smooth fiber of $\pi$, that is a del
Pezzo surface. If $W$ is a surface then $S_T$ is a (not necessarily
minimally) ruled surface and $\pi$ is a conic bundle structure. In the latter case if $\pi$ has a section it is
easy, see for instance the proof of Corollary \ref{cor:cb}, to prove that $(T,S_T)$ is
birational to a cone in $\P^3$.
\end{remark}

\section{Cremona equivalence for cones}\label{sec:cones}
Here I am interested in cones in $\P^n$. Let $S\subset \P^n$ be an
irreducible reduced divisor of degree $d$  with a point $p$ of
multiplicity $d$. Let $H$ be a hyperplane in $\P^n\setminus \{p\}$ and
$C=S\cap H$. Then $S$ can be viewed as the cone over the variety
$C$.  
It is easy to see that if $C_1, C_2\subset\P^{n-1}$ are Cremona
equivalent divisors then the cones over them are  Cremona
equivalent.

\begin{proposition}\label{pro:easy} Let $C_1$ and $C_2$ be Cremona
  equivalent divisors in $\P^{n-1}$ and  $S_1$, $S_2$ cones over them
  in $\P^n$. Then $S_1$ is
  Cremona equivalent to $S_2$.
\end{proposition}
\begin{proof}
  Without loss of generality I may assume that $S_1$ and $S_2$ have the
  same vertex in the point $[0,\ldots,0,1]$ and $C_1\cup C_2\subset (x_n=0)$. Let $\H\subset|\O_{\P^{n-1}}(h)|$
  be a linear system realizing the Cremona equivalence between $C_1$ and
  $C_2$. 
Hence I have a  Cremona map  $\psi:\P^{n-1}\rat\P^{n-1}$ given by a $n$-tuple $\{f_0,\ldots,f_{n-1}\}$,
with $f_i\in k[x_0,\ldots,x_{n-1}]_h$. This allows me to produce a map $\Psi:\P^n\rat\P^n$ considering the linear system 
$$\H^\prime:=\{f_0,\ldots,f_{n-1},x_nx_0^{h-1}\}. $$
Note that the general element in $\H^\prime$ has 
multiplicity $h-1$ at the point $p$ and, in the chosen base, there is
only one element of
multiplicity exactly $h-1$.  This shows that lines through $p$ are
sent to lines through a fixed point. Moreover the restriction 
$\Psi_{|(x_n=0)}$ is the original map $\psi$. Hence the map $\Psi$ is
birational and realizes the required Cremona equivalence.
\end{proof}

Next I want to understand what happens if $C_1, C_2\subset\P^{n-1}$ are simply
birational as abstract varieties. To do this I need to produce 
``nice'' good models of the pairs $(\P^n,S_1)$ and $(\P^n,S_2)$.

Let 
$C\subset H\subset\P^n$ be a codimension two subvariety and
$S$ be a cone with vertex
$p\in\P^n\setminus H$ over $C$. I produce a good model of the
pair $(\P^n,S)$ as follows. First I blow up $p$ producing a morphism
$\epsilon:Y\to \P^n$ with exceptional
divisor $E$. Note that $Y$ has a scroll structure $\pi:Y\to\P^{n-1}$ given by
lines through $p$, and $S_Y$,
the strict transform, is just $\pi^*(C)$. Let $\nu:W\to\P^{n-1}$ be a
 resolution of the singularities of $C$ and take the fiber product
$$\xymatrix{
Z\ar[d]_{\pi_W}\ar[rr]^{\nu_Y}&&Y\ar[d]^{\pi}\\
W\ar[rr]^{\nu}&&\P^{n-1}.}
$$
Then the strict transform $S_Z$ is a smooth divisor and $(Z,S_Z)$ is a
good model of $(\P^n,S)$. Note that the threshold $\rho(Z,S_Z)$
vanishes. According to the $\sharp$-MMP philosophy this forces us to  produce different good models.

The following Lemma is probably well known, but I
prefer to state it, and prove it, to help the reader.
\begin{lemma}
  \label{lem:curve}
Let $C\subset X$ be an irreducible and reduced subvariety. Assume that
there exists a birational map $\chi:X\rat Y$ such that $\chi$ is an isomorphism on the generic point of $C$. Let $D:=\chi(C)$ be the image
and $X_C$, respectively $Y_D$ the blow up of $C$ and $D$ with morphism
$f_C$, $f_D$, and
exceptional divisors $E_C$, $E_D$ respectively. Then there
is a birational map $\chi_C:X_C\rat Y_D$ mapping $E_C$ onto $E_D$. In
other words $(X_C,E_C)$ is birational to $(Y_D,E_D)$.    
\end{lemma}
\begin{proof}Let $U\subset X$ be an open and dense subset intersecting
  $C$ such that $\chi_{|U}$ is an isomorphism. Then considering the fiber
  product
$$\xymatrix{
X_C\supset U_C\ar[d]_{f_{U}}\ar[rr]^{\chi_C}&&Y_D\ar[d]^{f_D}\\
X\supset U\ar[rr]^{\chi_{|U}}&&Y,}
$$
I conclude, by the Universal Property of Blowing Up, the existence of
the morphism $\chi_C$ with the required properties.
\end{proof}
\begin{remark}
  Let me stress that the above result is in general not true with the
  weaker assumption that  $\chi$ is a morphism  on the general point of
  $W$. On the other hand if $Y$ is
  $\P^n$ the statement can be rephrased also in this weaker form.
\end{remark}

Let us go back to the pair $(Z,S_Z)$. Let $\Gamma_Z$ be the strict
transform of a general hyperplane section of $S$. Then I may consider the
following elementary transformation of the scroll structure $\pi_W$
$$\xymatrix{
&Bl_{\Gamma_Z} Z\ar[dl]_{\gamma}\ar[dr]^{\eta}&\\
Z\ar@{.>}[rr]^{\Phi}\ar[d]^{\pi_W}&&V\ar[d]\\
W\ar[rr]^{\iso}&&W,}
$$
where $\gamma$ is the blow up of $\Gamma_Z$ and $\eta$ is the blow down
of the strict transform of $S_Z$ to a codimension 2 subvariety, say $\Gamma_S$.
Then there is a birational map $\f:V\rat\P^n$ sending $\Gamma_S$ to a
codimension 2 subvariety, say $\Gamma^\prime $, and such that $\f$ is an isomorphism on the generic point
of $\Gamma_S$. The map $\f$ can be easily constructed again via an
elementary transformation of the scroll structure followed by blow
downs of  exceptional divisors. 
We may summarize the above construction in the following proposition
and diagram.
\begin{proposition}
  \label{pro:con}
Let $S\subset\P^n$ be a cone and $C$ a general hyperplane section. Then
there are birational maps $\epsilon:\P^n\rat Y$ and $\eta:Y\rat V$ such that
$\epsilon_*S=:E$  and $\eta$ is the blow down of $E$ to a codimension
2 subvariety $\Gamma$. In particular $S$ is the valuation
associated to the ideal $\I_\Gamma$. Moreover there is a third
birational map $\f:V\rat \P^n$ sending $\Gamma$ to a codimension 2 subvariety
$C^\prime$ and such
that again $S$ is the valuation associated to $\I_{C^\prime}$
$$\xymatrix{
&Y\ar[dr]_{\eta}\supset E&\\
&&V\supset\Gamma\ar@{.>}[d]^{\f}\\
\P^n\supset S\ar@{.>}[rr]^{\Phi}\ar@{.>}[uur]^{\epsilon}&&\P^n\supset C^\prime}
.$$
The composition $\Phi:\P^n\rat\P^n$ is a birational map sending $S$ to
the codimension 2 subvariety $C^\prime$ and such that $S$ is the valuation associated to
the ideal $\I_{C^{\prime}}$.
\end{proposition}

We are now ready to prove the main result on cones in $\P^n$.
\begin{theorem}
  \label{th:cone} Let $S_1, S_2\subset\P^n$ be cones and $C_1$, $C_2$ general
  hyperplane sections. If $C_1$ is birational to $C_2$ then $S_1$ is Cremona
  equivalent to $S_2$. 
In particular all divisorial cones over a rational variety are Cremona equivalent to a hyperplane.
\end{theorem}
\begin{remark}
I doubt the other direction is true. Let $X\subset\P^n$ be a non
rational but stably
rational variety. Assume that $X\times\P^a$ is
rational but $X\times\P^{a-1}$ is not rational for some $a\geq
1$. Then $X\times\P^a$ can be birationally embedded as a cone, say $S$, with hyperplane
section birational to $X\times\P^{a-1}$. In principle $S$ could be
Cremona equivalent to a hyperplane but its hyperplane section cannot
be rational. This cannot occur for surfaces, see Corollary \ref{cor:3}.    
\end{remark}
\begin{proof}
Let $S_1$ and $S_2$ be two cones and $C_1$,  respectively, $C_2$ general hyperplane
sections. 
Then by Proposition \ref{pro:con} there are
birational maps $\Phi_i:\P^n\rat\P^n$ sending $S_i$ to a codimension 2
subvariety 
$C^\prime_i$ and such that:
\begin{itemize}
\item  $S_i$ is the valuation
associated to $\I_{C^\prime_i}$,
\item $C^\prime_i$ is birational to $C_i$.
\end{itemize} 
I am assuming that  $C_1$ is birational to $C_2$. Then the main result and its proof \cite[p. 92]{MP} states that there is a Cremona map
$\chi:\P^n\rat\P^n$ sending $C^\prime_1$ to $C^\prime_2$  and such that
$\chi$ is an isomorphism on the generic point of $C^\prime_1$.
Then by Lemma \ref{lem:curve} I may extend $\chi$ to the blow up of
the $C^\prime_i$ to produce the required Cremona equivalence $\Psi$
$$\xymatrix{
&Y_1\ar[dr]_{\eta_1}\ar@{.>}[rrr]^{\Psi}\supset E_1&&&Y_2\ar[dl]^{\eta_2}\supset E_2\\
&&V_1\supset\Gamma_1\ar@{.>}[d]^{\f_1}&V_2\supset\Gamma_2\ar@{.>}[d]_{\f_2}&\\
\P^n\supset S_1\ar@{.>}[rr]^{\Phi_1}\ar@{.>}[uur]^{\epsilon_1}&&\P^n\supset
C_1^\prime\ar@{.>}[r]^\chi&\P^n\supset C_2^\prime&&\P^n\supset S_2\ar@{.>}[ll]_{\Phi_2}\ar@{.>}[uul]_{\epsilon_2}}
$$
\end{proof}

As observed before the result can be strengthened in lower dimension.

\begin{corollary}\label{cor:3}
  Let $S_1, S_2\subset\P^3$ be cones and $C_1$, $C_2$ general
  hyperplane sections. Then $C_1$ is birational to $C_2$ if and only
  if  $S_1$ is Cremona
  equivalent to $S_2$. 
In particular rational cones  are Cremona equivalent to a plane.
\end{corollary}
\begin{proof}
  I have to prove that if $S_1$ and $S_2$ are Cremona equivalent then
  also $C_1$ and $C_2$ are birational. Note that the irregularity of a
  resolution of $S_i$ is a birational invariant and it is the
  geometric genus of the curve $C_i$. This yields $g(C_1)=g(C_2)$ and
  concludes the proof.
\end{proof}

\begin{remark}
As observed, in the special case of rational surfaces in $\P^3$ this gives the
Cremona equivalence to a plane for any rational cone.
In particular any rational cone has positive sup-threshold.
\end{remark}

It remains to translate the statement in $\sharp$-MMP dictionary for
conic bundles.

\begin{corollary}\label{cor:cb} Let $S\subset\P^3$ be a rational surface. Assume that there
  is:
  \begin{itemize}
  \item[a)] a $\sharp$-minimal model of the pair $(T,S_T)$ such that $T$ has
  a conic bundle structure $\pi:T\to W$  onto a rational surface $W$,
  $S_T=\pi^*C$, for a curve $C\subset W$,
\item [b)] a birational  map $\chi:T\rat \P^3$ that contracts $S_T$ to
  a curve,
  say $\Gamma$, such that $S_T$ is the valuation associated to $\I_\Gamma$. 
  \end{itemize}
Then $\overline{\rho}(T,S_T)=\overline\rho(\P^3,S)>0$.
Assumption b) is always satisfied if $\pi$ has a section, i.e. if $\pi$ is a scroll structure.
\end{corollary}
\begin{proof}
  Let $(T,S_T)$ be a good model as in assumption a).
 By hypothesis there
is a birational map $\chi:T\rat \P^3$ that contracts $S_T$ onto a
curve $\Gamma$ and such that $S_T$ is the valuation associated to
$\I_\Gamma$.  The curve $\Gamma$ is dominated by a rational surface,
and it is therefore rational. The extension trick used in Theorem \ref{th:cone}
yields that $(T,S_T)$ is birationally equivalent to a plane in $\P^3$.
This is enough to prove that $\overline\rho(\P^3,S)>0$.

If $\pi$ has a section then all fibers are irreducible. Let
$\phi:T\rat Y$ be an elementary transformation that blows down
$S_T$ to a curve $\Gamma$. Then $Y$ has a scroll structure onto $W$ and I  may run
a $\sharp$-MMP on the base surface $W$, as described in
\cite[p. 700]{Me}, that is an isomorphism in a neighborhood of $\Gamma$.
This yields a new 3-fold model $Z$ with a Mori fiber space
structure onto either $\P^2$ or a ruled surface and then via
elementary transformation of the scroll structure I produce the
required map $\chi$.
\end{proof}
\begin{remark}
Unfortunately the birational geometry of rational conic bundles
without sections is
very poorly understood and it is difficult to understand whether
condition b) is always satisfied or not, even assuming the standard conjectures \cite{Isk}.
\end{remark}

\end{document}